\documentclass[11pt]{amsart}
\newtheorem{theorem}{Theorem}[section]
\newtheorem{lemma}[theorem]{Lemma}

\newtheorem*{thma}{Theorem A}
\theoremstyle{definition}

\theoremstyle{remark}

\numberwithin{equation}{section}

\def\R{{\mathbb R}}

\def\intslash{\rlap{\kern  .32em $\mspace {.5mu}\backslash$ }\int}
\def\qsl{{\rlap{\kern  .32em $\mspace {.5mu}\backslash$ }\int_{Q_x}}}

\def\emph#1{{\it #1 }}

\def\supp{{\text{\rm supp}}}

\def\rta{\rightarrow}

\def\alp{\alpha}

\def\eps{\varepsilon}

\def\lam{\lambda}

\def\om{\omega}              
\def\fr{\frac}
\newcommand{\Be}{\begin{equation}}
\newcommand{\Ee}{\end{equation}}
\newcommand{\Bes}{\begin{equation*}}
\newcommand{\Ees}{\end{equation*}}
\newcommand{\Bsp}{\begin{split}}
\newcommand{\Esp}{\end{split}}
\newcommand{\Bm}{\begin{multline}}
\newcommand{\Em}{\end{multline}}
\newcommand{\Bea}{\begin{eqnarray}}
\newcommand{\Eea}{\end{eqnarray}}
\newcommand{\Beas}{\begin{eqnarray*}}
\newcommand{\Eeas}{\end{eqnarray*}}
\newcommand{\Benu}{\begin{enumerate}}
\newcommand{\Eenu}{\end{enumerate}}
\newcommand{\Bi}{\begin{itemize}}
\newcommand{\Ei}{\end{itemize}}
\begin{document}

\title{Weighted bound for commutators}


\author{Yong Ding}
\address{School of Mathematical Sciences, Beijing Normal University, Beijing 100875, People's Republic of China }
\email{dingy@bnu.edu.cn}
\thanks{The work is supported by NSFC (No. 11371057) and  SRFDP (No. 20130003110003).}

\author{Xudong Lai}
\address{School of Mathematical Sciences, Beijing Normal University, Beijing 100875, People's Republic of China}
\email{xudonglai@mail.bnu.edu.cn}


\date{October 5, 2013}



\begin{abstract}
Let $K$ be the Calder\'on-Zygmund convolution kernel on $\mathbb{R}^d (d\geq2)$. Define the commutator associated with  $K$ and $a\in L^\infty(\mathbb{R}^d)$ by
\[ T_af(x)=p.v. \int K(x-y)m_{x,y}a\cdot f(y)dy. \]
Recently, Grafakos and Honz\'{\i}k \cite{GH} proved that $T_a$ is of weak type (1,1) for $d=2$. In this paper, we show
that $T_a$ is also weighted weak type (1,1) with the weight $|x|^\alpha\,(-2<\alpha <0)$ for $d=2$. Moreover, we prove that
$T_a$ is bounded on weighted $L^p(\mathbb{R}^d)\,(1<p<\infty)$ for all $d\ge2$.
\end{abstract}

\maketitle

\section{Introduction}
Suppose that $K$ is the Calder\'on-Zygmund convolution kernel on $\mathbb{R}^d\setminus\{0\}\, (d\ge2)$,  which means that $K$ satisfies following three conditions:
\begin{equation}\label{e:K_1}
|K(x)|\leq C|x|^{-d},
\end{equation}
\begin{equation}\label{e:K_2}
\int_{R<|x|<2R}K(x)dx=0, \text{ for all $R>0$},
\end{equation}
\begin{equation}\label{e:K_3}
|\nabla K(x)|\leq \frac{C}{|x|^{d+1}}.
\end{equation}

In 1987, Christ and Journ\'e \cite{CJ} introduced a commutator associated with $K$ and $a\in L^\infty(\mathbb{R}^d)$ by
\Bes
T_af(x)={\rm p.v.}\int K(x-y)m_{x,y}a\cdot f(y)dy,\qquad \text{ $f\in \mathcal{S}(\mathbb{R}^d)$,}
\Ees
where $\mathcal{S}(\mathbb{R}^d)$ denotes the Schwartz class and $$m_{x,y}a=\int_0^1a((1-t)x+ty)dt.$$  Note that when $d=1$, then
\[m_{x,y}a=\frac{\int_0^xa(z)dz-\int_0^ya(z)dz}{x-y}.\]
In this case, let $K(x)=\fr{1}{x}$ and $A(x)=\int_0^xa(z)dz$, then $A'(x)=a(x)\in L^\infty(\mathbb{R})$. So
\[T_af(x)={\rm p.v.}\int_\mathbb{R}\frac{A(x)-A(y)}{x-y}\frac{f(y)}{x-y}dy,\]
which is the famous Calder\'on commutator discussed in \cite{Cal}.

In \cite{CJ},  Christ and Journ\'e  showed that $T_a$ is bounded on $L^p(\mathbb{R}^d)$
 for $1<p<\infty$.
In 1995,    Hofmann \cite{H1} gave the weighted $L^p(\mathbb{R}^d)\,(1<p<\infty)$ boundedness of $T_a$ when the kernel $K(x)=\Omega(x/|x|)|x|^{-d}$.
 Recently, Grafakos and Honz\'{\i}k \cite{GH} proved that $T_a$ is weak type $(1,1)$ for
$d=2$. Further, Seeger \cite{S} showed that $T_a$ is still weak type $(1,1)$ for all $d\ge2$. The purpose of this paper is to establish a weighted variety of Grafakos and Honz\'{\i}k's results in \cite{GH}. In the sequel, for $1\leq p\leq \infty$, $A_p$ denotes the  Muckenhoupt weight class  and $L^p(\om)$ denotes the weighted $L^p(\mathbb{R}^d)$ space with norm $\|\cdot\|_{p,\omega}$. We also denote $\omega(E)=\int_E\om(x) dx$ for a measurable set $E$ in $\mathbb{R}^d$. The main result obtained in the present paper is as follows.

\begin{theorem}\label{t:2}Suppose $K$ satisfies $(\ref{e:K_1}), (\ref{e:K_2})$
and $(\ref{e:K_3})$ for $d=2$.  Let
$a\in L^\infty(\mathbb{R}^2)$ and $\omega(x)=|x|^\alpha$ for $-2<\alpha <0$. Then there exists a constant $C>0$ such that
\[ \omega(\{x\in \mathbb{R}^2:|T_af(x)|>\lambda\})\leq C\lambda^{-1}\|a\|_\infty\|f\|_{1, \omega} \]
for all $\lambda>0$ and $f\in L^1(\om)$.
\end{theorem}

 We would like to point out that the proof of Theorem $\ref{t:2}$ follows the nice idea from \cite{GH}.  However, there are some differences in proving $T_a$ is of weak type $(1,1)$ for the weighted case. In fact, the essential difficulties of proving Theorem $\ref{t:2}$ are to show the  smoothness of kernels of $(T_j^*T_j)_\om$ and $(T_i^*T_j)_\om$ (see (\ref{e:hold smooth}) and (\ref{e:hold}) below, respectively), these estimates are more complicated than no weight case, although we only consider power weight $|x|^{\alp}$ for $-2<\alpha\leq 0$. Our main innovations are further decomposition of power weight according to the dyadic decomposition. Note that $|x|^\alpha\in A_1(\mathbb{R}^2)$ if and only if $-2<\alpha\leq 0$, but our method cannot be used to deal with the general $A_1$ weight. This is the reason why we now cannot get a similar result as  Theorem $\ref{t:2}$ for general weight $w\in A_1(\mathbb{R}^2)$.

In order to prove Theorem $\ref{t:2}$, we need to establish the weighted $L^p$ boundedness of $T_a$
(actually we only need weighted $L^2$ boundedness).
Although the $L^p(\om)$ boundedness of $T_a$ given by \cite[ Theorem 2.15]{H1} for the homogeneous kernel   $K(x)=\Omega(x/|x|)|x|^{-d}$,
it seems that one cannot apply directly to $T_a$ with the kernel satisfying (\ref{e:K_1})-(\ref{e:K_3}) discussed in this paper.
However, Hofmann established a weighted $L^p$ boundedness criteria in \cite{H1} which is similar to $T1$ theorem.
The proof of Theorem \ref{t:1} given here is an application of that criteria. More precisely, $T_a$ is a special example of
the general operators studied in \cite{H1}.
\begin{theorem}\label{t:1}
Suppose $K$ satisfies the conditions $(\ref{e:K_1}), (\ref{e:K_2})$
and $(\ref{e:K_3})$. Let $a\in L^\infty(\mathbb{R}^d)\,(d\ge2)$ and $\omega\in A_p$, $1<p<\infty$. Then there exists a constant $C>0$ such that
\begin{equation}
\|T_af\|_{p,\omega}\leq C\|a\|_\infty\|f\|_{p,\omega}.
\end{equation}

\end{theorem}

This paper is organized as follows.  The proof of Theorem \ref{t:1} is given in Section 4.
In Section 2, we complete the proof of Theorem \ref{t:2} based on Theorem \ref{t:1} and Lemma \ref{e:main}. Moreover, in this section, we also state that the proof of Lemma \ref{e:main}  can be reduced to two key lemmas, their proofs will be given in Section 3. Throughout this paper the letter $C$ will stand for a positive constant which is independent of the essential variables and not necessarily the same one in each occurrence.

\section{Proof of Theorem \ref{t:2}}
Let us begin by giving  an analogous Calder\'on-Zygmund decomposition of
$f\in L^1(\omega)$. First, we recall the Whitney decomposition which can be found in \cite{L}:
\begin{lemma}$(\mathbf{Whitney\ decomposition})$
Let $F$ be an open nonempty proper subset of $\mathbb{R}^d$. Then there exists a family of dyadic closed cubes
$\{Q_j\}_j$ such that

{\rm (a)}\ $\bigcup Q_j=F$ and $Q_j$'s have disjoint interior.

{\rm (b)}\  $\sqrt{d}\cdot l(Q_j)\leq dist(Q_j,F^c)\leq 4\sqrt{d}\cdot l(Q_j)$, where $l(Q_j)$ denotes the side's length of $Q_j$.
\end{lemma}

\begin{lemma}\label{l:CZ}
Let $\omega\in A_1$ and $f\in L^1(\om)$. Set $E:=\{Mf(x)>\frac{\lambda}{\|a\|_\infty}\}$
where $M$ is the Hardy-Littlewood maximal operator.
Then for $a\in L^\infty(\mathbb{R}^d)$ and $\lam>0$, we have the following conclusions:

{\rm (i)}\ $E=\bigcup\limits_n Q_n$, ${Q_n}$'s are disjoint dyadic cubes.

{\rm (ii)}\ $\omega(E)\leq C\frac{\|a\|_\infty}{\lambda}\|f\|_{1,\omega}$.

{\rm (iii)}\ $f=g+b$.

{\rm (iv)}\ $b=\sum b_n$, $\supp b_n\subset Q_n$, $\int b_n=0$, $\|b_n\|_{1}\leq C\frac{\lambda}{\|a\|_\infty} |Q_n|$, $\|b\|_{1,\omega}\leq C\|f\|_{1,\omega}$.

{\rm (v)}\ $\|g\|^2_{2,\omega}\leq C\frac{\lambda}{\|a\|_\infty}\|f\|_{1,\omega}$.
\end{lemma}
\begin{proof}
Since $E$ is open, we can make a dyadic Whitney decomposition of the set $E$. Thus $E$ is the union of the disjoint dyadic cubes $Q_n$ and we have
\begin{equation}\label{e:whitney}
\sqrt{d}\,l(Q_n)<dist(Q_n,E^c)<4\sqrt{d}\,l(Q_n).
\end{equation}
By the weighted weak type (1,1) of $M$, we have
\Be\label{e:ome}
\omega (E)\leq C\frac{\|a\|_\infty}{\lambda}\|f\|_{1,\omega}.
\Ee
We write $f=g+b$, where $g=f\chi_E^c+\sum\limits_n\frac{1}{|Q_n|}\int_{Q_n}f(x)dx\chi_{Q_n}$, $b=\sum\limits_n \{f-\frac{1}{|Q_n|}\int_{Q_n}f(x)dx\}\chi_{Q_n}=:\sum\limits_nb_n$.
So, $b_n$ supports in $Q_n$ and $\int b_n=0$. Let $tQ_n$ denote the cube with $t$ times the side length of $Q_n$ and the same center.
We first claim that
\Be\label{e:czf}
\frac{1}{|Q_n|}\int_{Q_n}|f(x)|dx\leq C\frac{\lambda}{\|a\|_\infty}.
\Ee
In fact, by the Whitney decomposition's property (\ref{e:whitney}) we have $9\sqrt{d}Q_n\cap E^c\neq{\O}$. Thus by the
 definition of $E$, there exists $x_0\in 9\sqrt{d}Q_n$ such that $Mf(x_0)\leq\frac{\lambda}{\|a\|_\infty}$. Using the
 property of maximal function, we have $\fr{1}{|9\sqrt{d}Q_n|}\int_{9\sqrt{d}Q_n}|f(x)|dx\leq
 C\frac{\lam}{\|a\|_\infty}$. Hence we have the estimate $$\frac{1}{|Q_n|}\int_{Q_n}|f(x)|dx\leq
 \frac{1}{|Q_n|}\int_{9\sqrt{d}Q_n}|f(x)|dx\leq C\frac{\lambda}{\|a\|_\infty}.$$
For $b_n$ and $b$, by (\ref{e:ome}) and (\ref{e:czf}) we have
$$\|b_n\|_{1}\leq 2\int_{Q_n}|f(x)|dx\leq C\frac{\lambda}{\|a\|_\infty} |Q_n|,$$
$$\|b\|_{1,\omega}\leq \|f\|_{1,\omega}+C\frac{\lambda}{\|a\|_\infty}\omega(E)\leq C\|f\|_{1,\omega}.$$

Note that if $x\in E^c$, it is obvious that $|f(x)|\leq \fr{\lam}{\|a\|_\infty}$. Using this fact, (\ref{e:ome}) and (\ref{e:czf}), we have
$$\|g\|^2_{2,\omega}\leq \frac{\lambda}{\|a\|_\infty}\|f\|_{1,\omega}+C\big(\frac{\lambda}{\|a\|_\infty}\big)^2\omega(E)
\leq C\frac{\lambda}{\|a\|_\infty}\|f\|_{1,\omega}.$$
\end{proof}
In the following we use Lemma \ref{l:CZ} for $d=2$ and $\omega(x)=|x|^{\alpha}$ with $-2<\alpha<0$. Denote $\mathfrak{Q}_k=\{Q_n: l(Q_n)=2^k \}$ and let $B_k=\sum\limits_{Q\in\mathfrak{Q}_k} b_Q$. Taking a smooth function $\phi$ on $[0,\infty)$ such that supp$\phi\subset\{x:\frac{1}{4}\leq |x|\leq 1\}$ and
$\sum_j\phi_j(x)=1$ for all $x\in \mathbb{R}^2\backslash\{0\}$, where $\phi_j(x)=\phi(2^{-j}x)$. Write $K=\sum_j
K_j$, where $K_j(x)=\phi_j(x)K(x)$ and define the corresponding operators $T_j$ with the kernel $K_j(x-y)m_{x,y}a$. Clearly we have $T_a=\sum_jT_j$.

We now state a lemma, which plays an important role in the proof of  Theorem \ref{t:2}:
\begin{lemma}\label{e:main}
There exists an $\varepsilon>0$ such that for any integer $s\geq10$,
\begin{equation}
\big\|\sum_jT_jB_{j-s}\big\|_{2,\omega}^2\leq C2^{-\varepsilon s}\lambda \|a\|_\infty \|b\|_{1,\omega},
\end{equation}
where $C$ is a constant depended on $K$ only.
\end{lemma}
The proof of Lemma \ref{e:main} will be stated below. We now explain that Theorem \ref{t:2} can be obtained by Lemma \ref{e:main} and Theorem \ref{t:1}. In fact, for any $f\in L^1(\om)$ and $\lam>0$, by Lemma \ref{l:CZ}, we have
\[\omega(\{|T_af(x)|>\lambda\})\leq \omega\big(\{|T_ag(x)|>\lambda/2\}\big)+\omega\big(\{|T_ab(x)|>\lambda/2\}\big).\]
Since $g\in L^2(\om)$, by Theorem \ref{t:1}, we have $\|T_ag\|_{2,\omega}\leq C\|a\|_\infty\|g\|_{2,\omega}$. Hence, by
Chebychev's inequality and Lemma \ref{l:CZ},
\[\omega(\{|T_ag(x)|>\lambda/2\})\leq 4\|T_ag\|_{2,\omega}^2/{\lambda ^2}\leq C\frac{\|a\|_\infty^2\lambda\|f\|_{1,\omega}}{\|a\|_\infty
\lambda^2}=C\|a\|_\infty\frac{\|f\|_{1,\omega}}{\lambda}.\]
Let $E^*=\bigcup2^{11}Q_n$. Then we have
\[\omega(\{|T_ab(x)|>\lambda/2\})\leq\omega(E^*)+\omega(\{x\in (E^*)^c:|T_ab(x)|>\lambda/2\}).\]
Since $\omega$ satisfies the doubling condition, the set $E^*$ satisfies
\begin{equation}\label{e:weak}
\omega(E^*)\leq C\omega(E)\leq C\frac{\|a\|_\infty}{\lambda}\|f\|_{1,\omega}.
\end{equation}
We write
\[T_ab(x)=\sum_{s\in\mathbb{Z}}\sum_{j\in\mathbb{Z}}T_jB_{j-s}.\]
Note that $T_jB_{j-s}(x)=0$, for $x\in (E^*)^c$ and $s<10$. Therefore
\[\omega\bigg(\bigg\{x\in (E^*)^c:T_ab(x)>\frac{\lambda}{2}\bigg\}\bigg)=\omega\bigg(\bigg\{x\in (E^*)^c:
\bigg|\sum_{s\geq10}\sum_{j\in \mathbb{Z}}T_jB_{j-s}(x)\bigg|>\frac{\lambda}{2}\bigg\}\bigg).\]
From Lemma \ref{e:main} we get
\[\big\|\sum_{s\geq 10}\sum_{j\in \mathbb{Z}}T_jB_{j-s}\big\|_{2,\omega}^2\leq\big(\sum_{s\geq10}\big\|\sum_{j\in \mathbb{Z}}T_jB_{j-s}\big\|_{2,\omega}\big)^2\leq C\lambda \|a|\|_\infty \|b\|_{1,\omega}.\]
By Chebychev's inequality, we have
\[\omega\bigg(\bigg\{x\in (E^*)^c:{\bigg|\sum_{s\geq10}}\sum_{j\in \mathbb{Z}}T_jB_{j-s}(x)\bigg|>\frac{\lambda}{2}\bigg\}\bigg)\leq
C\|a\|_\infty\frac{\|b\|_{1,\omega}}{\lambda}\leq C\|a\|_\infty\frac{\|f\|_{1,\omega}}{\lambda}.\]
Hence we get the conclusion of Theorem \ref{t:2}. Thus, to complete the proof of Theorem {\ref{t:2}}, it suffices to show
Lemma \ref{e:main} and Theorem \ref{t:1}. The proof of Theorem \ref{t:1} will be given in Section 4. Let us first
state Lemma \ref{e:main}. We write
\begin{equation*}
\begin{split}
&\mspace{20mu}\big\|\sum_{j\in\mathbb{Z}}T_jB_{j-s}\big\|_{2,\omega}^2=\sum_{i,j\in\mathbb{Z}}\langle T_jB_{j-s},T_iB_{i-s}\rangle_\omega
\\&=\sum_{j\in\mathbb{Z}}\|T_jB_{j-s}\|_{2,\omega}^2+2\sum_i\sum_{j=i-2}^{i-1}\langle
T_jB_{j-s},T_iB_{i-s}\rangle_\omega+
2\sum_{i\in\mathbb{Z}}\sum_{j\leq{i-3}}\langle T_jB_{j-s},T_iB_{i-s}\rangle_\omega\\
&=:I+II+III,
\end{split}
\end{equation*}
where $\langle u,v\rangle_\omega=\int u(x)v(x)\omega(x)dx$ for the  real valued functions $u$ and $v$.

Note that the estimate of $II$ can be reduced to $I$:
\begin{align*}
2\bigg|\sum_i\sum_{j=i-2}^{i-1}\langle T_jB_{j-s},T_iB_{i-s}\rangle_\omega\bigg|&\leq \sum_i\sum_{j=i-2}^{i-1}(\|T_jB_{j-s}\|_{2,\omega}^2+
\|T_iB_{i-s}\|_{2,\omega}^2)\\&\leq 4\sum_i\|T_iB_{i-s}\|^2_{2,\omega}.
\end{align*}
Hence, if we can establish the following lemma, then we may get the estimate of $I$ and $II$.
\begin{lemma}\label{e:diagonal}
There exists an $\varepsilon>0$ such that for any fixed $s\geq10$,
\begin{equation}\label{eq:diagonal}
\|T_jB_{j-s}\|_{2,\omega}^2\leq C2^{-\varepsilon s}\lambda\|a\|_\infty \|B_{j-s}\|_{1,\omega},
\end{equation}
where $C$ is a constant dependent on the properties of $K$.
\end{lemma}

To handle the cross terms $III$, we need the following conclusion:
\begin{lemma}\label{e:cross}
There exist $C$ , $\varepsilon>0$ such that
\[\bigg|\sum_{i\in\mathbb{Z}}\sum_{j\leq{i-3}}\langle T_jB_{j-s},T_iB_{i-s}\rangle_\omega\bigg|\leq C2^{-\varepsilon s}\lambda\|a\|_\infty\|b\|_{1,\om}\]
for any $s\geq10$.
\end{lemma}
So, to get Lemma \ref{e:main}, it remains to prove Lemma \ref{e:diagonal} and Lemma \ref{e:cross}, which will be given in the following section.

\vspace{4mm}

\section{Proofs of lemma \ref{e:diagonal} and lemma \ref{e:cross}}

\hspace{-0.7cm}{\bf 3.1 Proof of Lemma \ref{e:diagonal}}

First let us consider Lemma \ref{e:diagonal}.
For any $i,j\in \mathbb{Z}$, we write
\[\langle T_jB_{j-s},T_iB_{i-s}\rangle_\omega =\langle(T_i^*T_j)_\omega B_{j-s},B_{i-s}\rangle,\]
where $(T_i^*T_j)_\omega$ has the kernel
\Be\label{1e:k_ij}
K_{i,j}(y,x)=\int K_i(z-y)K_j(z-x)m_{x,z}a\cdot m_{y,z}a\cdot\omega(z)dz.
\Ee
Hence we can write
\[\|T_jB_{j-s}\|_{2,\om}^2=\langle(T_j^*T_j)_\omega B_{j-s},B_{j-s}\rangle,\]
It is easy to see that the following two lemmas are the key to proving Lemma \ref{e:diagonal}.
\begin{lemma}\label{1:f1}For $|y|>2^{j+1}$ or $|y|<2^{j-3}$, there exist  $C$ , $\varepsilon>0$ such that
\[|(T_j^*T_j)_{\omega}B_{j-s}(y)|\leq C2^{-\varepsilon s}\lambda\|a\|_\infty\omega(y)\]
for any integer $s\geq 10$.
\end{lemma}
\begin{lemma}\label{1:f2}For $2^{j-3}\leq|y|\leq 2^{j+1}$, there exist $C$ , $\varepsilon>0$  such that
\[|(T_j^*T_j)_{\omega}B_{j-s}(y)|\leq C2^{-\varepsilon s}\lambda\|a\|_\infty\omega(y)\]
for any integer $s\geq 10$.
\end{lemma}

\hspace{-0.7cm}\emph{Proof of Lemma}\ref{1:f1}: We claim that the kernel $K_{j,j}$ which is given by (\ref{1e:k_ij}) has the
 H\"older smoothness:
\begin{equation}\label{e:hold smooth}
|K_{j,j}(y,x)-K_{j,j}(y,x')|\leq C2^{\frac{19}{20}j}|x-x'|^{\frac{1}{20}}2^{-3j}\|a\|_\infty^2\omega(y),
\end{equation}
for any $|x-x'|\leq 2^{j-10}$. Once we establish (\ref{e:hold smooth}), we can get Lemma \ref{1:f1}. In fact, write
\begin{align*}
&\mspace{20mu}\Big|\int K_{j,j}(y,x)B_{j-s}(x)dx\Big|
=\Big|\sum_{Q_n\in \mathfrak{Q}_{j-s}}\int K_{j,j}(y,x)b_n(x)dx\Big|\\
&\leq\Big|\sum_{Q_n\in\mathfrak{Q}_{j-s}}\int_{|y-x|>10\cdot2^{\frac{9}{10}j}|x-x_{Q_n}|^{\frac{1}{10}}}(K_{j,j}(y,x)-K_{j,j}(y,x_{Q_n}))b_n(x)dx\Big|\\
&\mspace{20mu}+\Big|\sum_{Q_n\in\mathfrak{Q}_{j-s}}\int_{|y-x|<10\cdot2^{\frac{9}{10}j}|x-x_{Q_n}|^{\frac{1}{10}}}(K_{j,j}(y,x)-K_{j,j}(y,x_{Q_n}))b_n(x)dx\Big|\\
&=:J_1+J_2,
\end{align*}
where $x_{Q_n}$ denotes the center of $Q_n$. For $J_1$, by using the H\"older smoothness (\ref{e:hold smooth}) we have
\begin{align*}
J_1&\leq\sum_{Q_n\in \mathfrak{Q}_{j-s}}c2^{\frac{19}{20}j-3j}\|a\|_\infty^2\omega(y)\int |x-x_{Q_n}|^{\frac{1}{20}}|b_n(x)|dx\\
&\leq C2^{-2j}2^{-\frac{s}{20}}\omega(y)\lambda\|a\|_\infty\sum_{dist(Q_n,y)\leq2^{j+1}}|Q_n|\\
&\leq C2^{-\frac{s}{20}}\omega(y)\lambda\|a\|_\infty.
\end{align*}
For $J_2$, we have
\begin{equation*}
J_2\leq C2^{-2j}\omega(y)\lambda\|a\|_\infty\sum_{dist(Q_n,y)\leq10\cdot2^{j-\frac{s}{10}}}|Q_n|
\leq C2^{-\frac{s}{5}}\|a\|_\infty\lambda\omega(y).
\end{equation*}
Then we can choose $\varepsilon=\frac{1}{20}$.
To obtain the kernel's H\"older smoothness (\ref{e:hold smooth}), write
\[K_{j,j}(y,x)-K_{j,j}(y,x')=: A_1+A_2,\]
where
\begin{equation*}
A_1=\int (K_j(z-x)-K_j(z-x'))K_j(z-y)m_{x,z}a\cdot m_{y,z}a\cdot\omega(z)dz
\end{equation*}
and
\begin{equation*}
A_2=\int K_j(z-x')K_j(z-y)(m_{x,z}a-m_{x',z}a)m_{y,z}a\cdot\omega(z)dz.
\end{equation*}
Note that $2^{j-2}\leq|z-y|\leq2^j$. Then $|z|\geq C\cdot \max\{|y|,2^{j-3}\}$ when $|y|<2^{j-3}$ or $|y|>2^{j+1}$. Thus, $\omega(z)\leq C\omega(y)$. Since $K_j$ is a smooth function with compact support, we have
\[|K_j(z-x)-K_j(z-x')|=\Big|\int_0^1\langle x'-x,\nabla K_j(z-(1-s)x'-sx)\rangle ds\Big|\leq C2^{-3j}|x-x'|.\]
Therefore
\begin{equation}\label{e:A}
|A_1|\leq C2^{-3j}\|a\|_\infty^2\omega(y)|x-x'|.
\end{equation}
To estimate $A_2$, we switch to polar coordinates $z=y+r\theta$, then
\begin{equation}\label{e:A_1}
A_2=\int_{\mathbb{S}^1}\int_{2^{j-2}}^{2^j}\psi(r)(m_{x,y+r\theta}a-m_{x',y+r\theta}a)
\omega(y+r\theta)drd\theta,
\end{equation}
where
\begin{equation*}
\psi(r)=K_j(y-x'+r\theta)K_j(r\theta)m_{y,y+r\theta}a\cdot r=K_j(y-x'+r\theta)K_j(r\theta)\int_0^r a(y+s\theta)ds.
\end{equation*}
It is easy to see that $\|\psi\|_\infty\leq C2^{-3j}\|a\|_\infty$ and $\|\psi'\|_\infty\leq C2^{-4j}\|a\|_\infty$.

We first split the integral over $\mathbb{S}^1$ as a sum over the arc
\[\bigg|\theta\pm\frac{x-y}{|x-y|}\bigg|<t_0\]
and its complement, $t_0$ will be chosen later as $C2^{-\frac{1}{10}j}|x-x'|^{\frac{1}{10}}$. Therefore the part of the integral in (\ref{e:A_1}) over this arc is bounded by $C2^{-2j}\|a\|_\infty^2t_0\omega(y)$.

Now we reduce $A_2$ to estimate the part of the outer integral in (\ref{e:A_1}) over the set
\[\bigg|\theta\pm\frac{x-y}{|x-y|}\bigg|\geq t_0.\]
By a rotation, without loss of generality, we can assume that $\theta=(1,0)$,
\begin{align*}
N&=\int_{2^{j-2}}^{2^j}\psi(r)\int_0^1a(x+s((r,0)-x+y))\omega(y+r(1,0))dsdr\\
&\mspace{20mu}-\int_{2^{j-2}}^{2^j}\psi(r')\int_0^1a(x'+s((r',0)-x'+y))\omega(y+r'(1,0))dsdr'.
\end{align*}
We make a coordinate transform. For the first term in the above integral, set $$u=x_1+s(r-x_1+y_1)\quad \text{and}\quad v=x_2+s(-x_2+y_2),$$ then  $r=\frac{u-x_1}{v-x_2}(y_2-x_2)-y_1+x_1$. For the second term, let
$$u=x'_1+s(r'-x'_1+y_1)\quad \text{and}\quad v=x_2+s(-x'_2+y_2),$$ then $r'=\frac{u-x'_1}{v-x'_2}(y_2-x'_2)-y_1+x'_1$. Therefore, we have
\begin{equation*}
N=\iint_A\psi(r)a(u,v)\omega(y+r(1,0))\frac{dudv}{|x_2-v|}-
\iint_{A'}\psi(r')a(u,v)\omega(y+r'(1,0))\frac{dudv}{|x'_2-v|},
\end{equation*}
where $A$ is the triangle with vertices $\{y+(2^{j-2},0),y+(2^j,0),x\}$ and $A'$ is the triangle with vertices $\{y+(2^{j-2},0),y+(2^j,0),x'\}$.
By symmetric, we may assume $x_2>y_2$. Observe that
\[\iint_A\frac{dudv}{|x_2-v|}=\int_{y_2}^{x_2}\frac{3\cdot2^j}{4}\frac{x_2-v}{x_2-y_2}
\frac{dv}{x_2-v}=\frac{3}{4}2^j.\]
Now we assume that
$$|x-x'|\leq2^{j-10}\quad \text{and}\quad  |x-y|>10|x-x'|^\frac{1}{10}2^{\frac{9}{10}j}$$ and set
$t_0=10|x-x'|^\frac{1}{10}2^{-\frac{1}{10}j}$. Since $\big|(1,0)\pm\frac{x-y}{|x-y|}\big|\geq t_0$ and
$t_0$ is small relative to
 $2^j$, we have $$\frac{|x_2-y_2|}{|x-y|}>\frac{1}{10}t_0=|x-x'|^{\frac{1}{10}}2^{-\frac{1}{10}j}.$$ Then
 $|x_2-y_2|\geq10|x-x'|^{\frac{1}{5}}2^{\frac{4}{5}j}$. By using an analogous method we obtain
 $|x'_2-y_2|\geq9|x-x'|^\frac{1}{5}2^{\frac{4}{5}j}$.
Using polar coordinate transform we get
\begin{align*}
\Big|\iint_{A\bigcap B(x,2^{\frac{3}{4}j}|x-x'|^\frac{1}{4})}\frac{dudv}{|x_2-v|}\Big|&\leq\iint_0^{2^{\frac{3}{4}j}
|x-x'|^\frac{1}{4}}
\frac{drd\theta}{|\sin\theta|}\\&\leq C2^{\frac{3}{4}j}|x-x'|^{\frac{1}{4}}\frac{2^j}{|x_2-y_2|}\leq C2^{\frac{19}{20}j}|x-x'|^{\frac{1}{20}},
\end{align*}
where the angle $\theta$ is between vector $v-x$ and $(1,0)$ and the second inequality comes from the geometry estimate $|\sin\theta|\geq C\frac{x_2-y_2}{2^j}$. So we have
\begin{equation}
\Big|\iint\limits_{A\bigcap B(x,2^{\frac{3}{4}j}|x-x'|^\frac{1}{4})}\psi(r)a(u,v)\omega(y+(r,0))\frac{dudv}{|x_2-v|}\Big|\leq C2^{-3j}\|a\|_\infty^2\omega(y)2^{\frac{19}{20}j}|x-x'|^{\frac{1}{20}}.
\end{equation}
Since $A'\bigcap B(x,2^{\frac{3}{4}j}|x-x'|^\frac{1}{4})\subseteq A'\bigcap B(x',2\cdot2^{\frac{3}{4}j}|x-x'|^\frac{1}{4})$, by an analogous method we also have the estimate
\begin{equation}
\bigg|\iint\limits_{A'\bigcap B(x,2^{\frac{3}{4}j}|x-x'|^\frac{1}{4})}\psi(r')a(u,v)\omega(y+(r',0))
\frac{dudv}{|x_2-v|}\bigg|\leq C2^{-\frac{41}{20}j}\|a\|_\infty^2\omega(y)|x-x'|^{\frac{1}{20}}.
\end{equation}

Now we denote $A'\triangle A=(A'\setminus A)\cup(A\setminus A')$. We claim that
\begin{equation}\label{e:sym}
\iint_{(A\triangle A')\setminus B(x,2^{\frac{3}{4}j}|x-x'|^\frac{1}{4})}\frac{dudv}{|x_2-v|}\leq C2^{\frac{19}{20}j}|x-x'|^{\frac{1}{20}}.
\end{equation}
Then by (\ref{e:sym}) we have
\[\bigg|\iint\limits_{(A\triangle A')\setminus B(x,2^{\frac{3}{4}j}|x-x'|^\frac{1}{4})}\psi(r)a(u,v)\omega(y+(r,0))\frac{dudv}{|x_2-v|}\bigg|\leq C2^{-\frac{41}{20}j}\|a\|_\infty^2\omega(y)|x-x'|^{\frac{1}{20}}.\]
Now we come back to prove ($\ref{e:sym}$) first in the case $x'_2=x_2$. By similar triangles, we obtain that
\[|(A'\triangle A)\cap\{(a_1,a_2):a_2=v\}|\leq\frac{2|x-x'||v-y_2|}{x_2-y_2}.\]
Then the integral in (\ref{e:sym}) has an estimate
\begin{equation}\label{e:AA'}
\iint\limits_{(A\triangle A')\setminus B(x,2^{\frac{3}{4}j}|x-x'|^\frac{1}{4})}\frac{dudv}{|x_2-v|}
\leq\int_{y_2}^{x_2-c_02^{\frac{3}{4}j}|x-x'|^{\frac{1}{4}}}\frac{2|x-x'|}{x_2-v}\frac{v-y_2}{x_2-y_2}dv,
\end{equation}
where $c_0$ is the minimum sine of the angle between vector $y+(2^{j-2},0)-x$ and $(1,0)$ and the angle between vector
$y+(2^j,0)-x$ and $(1,0)$. By a geometry estimate we have $c_0\geq C\frac{|x_2-y_2|}{2^j}$. So (\ref{e:AA'}) is controlled by
\[\frac{2|x-x'|}{|x_2-y_2|}\int_{c_o2^{\frac{3}{4}j}|x-x'|^{\frac{1}{4}}}^{x_2-y_2}
\frac{x_2-y_2-\upsilon}{\upsilon}d\upsilon
\leq C2^{\frac{19}{20}j}|x-x'|^{\frac{1}{20}}.\]

Now we consider the case where $x'_2>x_2$ or $x_2>x'_2$. By symmetry we only look at the case $x_2>x'_2>0$. We extend one of the sides of the shorter triangle $A'$ to make it have the same height as $A$.
Then we find a point $x''$ at the extend side such that $x_2=x''_2$.
Since
$$|x''-x'|<C2^{\frac{1}{5}j}|x-x'|^{\frac{4}{5}},$$ then
$$|x-x''|\leq C2^{\frac{1}{5}j}|x-x'|^{\frac{4}{5}}.$$
Replacing $A'$ by the larger triangle $A''$ with the vertex $\{x'',y+(2^{j-2},0),y+(2^j,0)\}$,
then $A\triangle A''$ contains $A\triangle A'$, and the ball
$B(x,2^{\frac{3}{4}j}|x-x'|^\frac{1}{4})\supset B(x,C2^{\frac{11}{16}j}|x-x'|^{\frac{5}{16}})$ for some constant $C$. Therefore, we have
\begin{equation}\label{e:AA''}
\iint_{(A\triangle A')\setminus B(x,2^{\frac{3}{4}j}|x-x'|^\frac{1}{4})}\frac{dudv}{|x_2-v|}\leq
\iint_{(A\triangle A'')\setminus B(x,c2^{\frac{11}{16}j}|x-x''|^\frac{5}{16})}\frac{dudv}{|x_2-v|}.
\end{equation}
Use the same method as the case $x_2=x'_2$, we can get that the right side of (\ref{e:AA''}) is bounded by
$C2^{\frac{19}{20}j}|x-x'|^{\frac{1}{20}}$.

The remaining part is $(A\cap A')\setminus B(x,2^{\frac{3}{4}j}|x-x'|^\frac{1}{4})$. By straightforward computation
 we have
\begin{equation}\label{e:rr'}
\begin{split}
|r-r'|&\leq |x_1-x_2|+\frac{|u-x'_1||x_2-x'_2||v-x'_2|}{|x_2-v\|x'_2-v|}\\
&\mspace{20mu}+\frac{|x_1-x'_1||y_2-x_2||v-x'_2|+|u-x'_1||y_2-x'_2||x_2-x'_2|}{|x_2-v||x'_2-v|}\\
&\leq C2^{\frac{19}{20}j}|x-x'|^\frac{1}{20}.
\end{split}
\end{equation}
Write
\begin{align*}
&\iint_{A\cap A'\setminus B(x,2^{\frac{3}{4}j}|x-x'|^{\frac{1}{4}})}\bigg[\frac{\psi(r)a(u,v)\omega(y+r\theta)}{|x_2-v|}
-\frac{\psi(r')a(u,v)\omega(y+r'\theta)}{|x'_2-v|}\bigg]dudv\\
&=:G_1+G_2+G_3,
\end{align*}
where
\begin{equation*}
\begin{split}
G_1&=\iint(\psi(r)-\psi(r'))a(u,v)\omega(y+r\theta)\frac{dudv}{|x_2-v|},\\
G_2&=\iint\psi(r')a(u,v)(\omega(y+r\theta)-\omega(y+r'\theta)\frac{dudv}{|x_2-v|},\\
\end{split}
\end{equation*}
and
\Bes
G_3=\iint\psi(r')a(u,v)\omega(y+r'\theta)\bigg(\frac{1}{|x_2-v|}-\frac{1}{|x'_2-v|}\bigg)dudv.
\Ees
For $G_1$, by ($\ref{e:rr'}$) and $|z|^\alpha=|y+r\theta|^\alpha\leq\omega(y)$, we have
\begin{align*}
|G_1|&\leq C2^{\frac{19}{20}j}|x-x'|^{\frac{1}{20}}\omega(y)\|a\|_\infty^2
\frac{1}{2^{4j}}\iint\frac{dudv}{|x_2-v|}\leq C2^{\frac{19}{20}j}|x-x'|^{\frac{1}{20}}\omega(y)\|a\|_\infty^22^{-3j}.
\end{align*}
For $G_2$, by ($\ref{e:rr'}$), $|y+r\theta|\geq C\cdot \max\{|y|,2^j\}$ and $|y+r'\theta|\geq C\cdot \max\{|y|,2^j\}$, we have
\begin{align*}
|G_2|&\leq C2^{\frac{19}{20}j}|x-x'|^\frac{1}{20}\|a\|_\infty^22^{-3j}\iiint_0^1|\nabla\omega(y+(sr+(1-s)r')\theta)|ds\frac{dudv}{|x_2-v|}\\
&\leq C2^{\frac{19}{20}j}|x-x'|^{\frac{1}{20}}2^{-3j}\|a\|_\infty^2\omega(y).
\end{align*}
For $G_3$, we also get
\begin{equation*}
|G_3|\leq C2^{\frac{19}{20}j}|x-x'|^{\frac{1}{20}}2^{-3j}\|a\|_\infty^2\omega(y).
\end{equation*}
Combining above estimates, we have
\begin{equation*}
|K_{j,j}(y,x)-K_{j,j}(y,x')|\leq C2^{\frac{19}{20}j}|x-x'|^{\frac{1}{20}}2^{-3j}\|a\|_\infty^2\omega(y)
\end{equation*}
for any $|x-x'|\leq 2^j\frac{1}{1000}$. Hence, we complete the proof of Lemma $\ref{1:f1}$.$\hfill{} \Box$

\vspace{3mm}
\hspace{-0.7cm}\emph{Proof of Lemma \ref{1:f2}}
\quad Fix  $2^{j-3}\leq |y|\leq2^{j+1}$. Let $K_j=K_j^1+K_j^2$,
where $$K_j^1(z-y)=K_j(z-y)\chi(z-y),$$ and $\chi$ is the characteristic function of the set
$\{x:\rho(x,y)\leq2^{-\beta s}\}$ with $\rho(x,y)=|\frac{y}{|y|}-\frac{x}{|x|}|$ and $\beta$ will be chosen later. Since $(T_j^*T_j)_\omega$ has kernel $K_{j,j}$, we write
$K_{j,j}=K_{j,j}^1+K_{j,j}^2,$ where
\[K_{j,j}^1(y,x)=\int K_j(z-x)K_j^1(z-y)m_{x,z}a\cdot m_{y,z}a\cdot\omega(z)dz\]
and
\[K_{j,j}^2(y,x)=\int K_j(z-x)K_j^2(z-y)m_{x,z}a\cdot m_{y,z}a\cdot\omega(z)dz.\]
Then we only need to prove Lemma \ref{1:f2} corresponding to $K_{j,j}^1$ and $K_{j,j}^2$. It is easy to check the
term corresponding to $K_{j,j}^1$. Indeed, we have
\begin{equation}
|K_{j,j}^1(y,x)|\leq C2^{-4j}\|a\|_\infty^2\int_{|z|\leq C|y|}\chi(z)|z+y|^{\alpha}dz\leq
c2^{-2j}\|a\|_\infty^2|y|^{\alpha}2^{-\beta s},
\end{equation}
where we use the following inequality ( see (0.3) in \cite[p.~425]{H2})
\begin{equation*}
\int_0^{a|y|}\big|s\frac{x}{|x|}+y\big|^\alpha ds\leq
\begin{cases}
C|y|^{\alpha+1}\big(1+|\frac{x}{|x|}+\frac{y}{|y|}|^{\alpha+1}\big) & \alpha\neq-1,\\ C\cdot\log^+|\frac{x}{|x|}+\frac{y}{|y|}| & \alpha=-1.
\end{cases}
\end{equation*}
Then the corresponding term
\begin{equation}
\Big|\int K_{j,j}^1(y,x)B_{j-s}(x)dx\Big|\leq C2^{-2j}\|a\|_\infty^2|y|^{\alpha}2^{-\beta s}
\int B_{j-s}(x)dx\leq C\|a\|_\infty\lambda\omega(y)2^{-\beta s}.
\end{equation}

Now we consider the remaining term corresponding to $K_{j,j}^2$. By the definition of
 $\chi$, $\rho(y-z,y)\geq2^{-\beta s}$. Since $\rho(y-z,y)=|\frac{y-z}{|y-z|}-\frac{y}{|y|}|\leq2\frac{|z|}{|y|}$, we have $|z|\geq C2^{-\beta s+j}$. Thus
\[\omega(z)\leq C2^{-\alpha\beta s}\omega(y).\]
Applying the same method as proving Lemma \ref{1:f1} we may obtain
\[\bigg|\int K_{j,j}^2(y,x)B_{j-s}(x)dx\bigg|\leq C\|a\|_\infty\lambda\omega(y)2^{\beta s(1-\alpha)-\frac{s}{5}}\]
as long as we choose $\beta=\frac{1}{10(1-\alpha)}$.
$\hfill{} \Box$

\vspace{2mm}

\hspace{-0.7cm}{\bf 3.2 Proof of Lemma \ref{e:cross}}

We write
\[\sum_{i\in\mathbb{Z}}\sum_{j\leq{i-3}}\langle T_jB_{j-s},T_iB_{i-s}\rangle_\omega =\sum_{i\in\mathbb{Z}}\big\langle\sum_{j\leq {i-3}}(T_i^*T_j)_\omega B_{j-s},B_{i-s}\big\rangle.\]
To prove Lemma \ref{e:cross}, it is easy to see that it suffices to prove the following lemma:
\begin{lemma}\label{e:cross1}
For a fixed $i$, there exist $C$ , $\varepsilon>0$ such that
\[\bigg|\sum_{j\leq{i-3}}(T_i^*T_j)_{\omega}B_{j-s}(y)\bigg|\leq C2^{-\varepsilon s}\lambda\|a\|_\infty\omega(y)\]
for any $s\geq10$.
\end{lemma}
\begin{proof}
The proof is similar to that of Lemma \ref{e:diagonal}. So we only give the difference. Consider two cases of
$y$ as Lemma \ref{1:f1} and Lemma \ref{1:f2} respectively. We first consider the case $|y|< 2^{j-3}$ or $|y|>2^{j+1}$. We
proceed with the proof as we do in Lemma \ref{1:f1}: For the analogous term of $A_2$, we switch to the polar coordinates
$z=y+r\theta$
\[A_2=\int_{\mathbb{A}}\int_{2^{i-2}}^{2^i}\psi(r)(m_{x,y+r\theta}a-m_{x',y+r\theta})\omega(y+r\theta)drd\theta,\]
where
\[\psi(r)=K_j(y-x'+r\theta)K_i(r\theta)m_{y,y+r\theta}a\cdot r\]
and $\mathbb{A}$ is an arc in $\mathbb{S}^1$.

We claim that $\mathbb{A}$ is an arc of length of about $2^{-i+j}$.
Indeed, consider the support of $K_j$ and $K_i$, we have
$$2^{j-2}\leq |y-x'+r\theta|\leq 2^j\qquad \text{and}\qquad 2^{i-2}\leq r\leq 2^i.$$
Since $j\leq i-3$, we get $2^{i-3}\leq|y-x'|\leq 2^{i+1}$. Let $\Theta$ be the smallest cone with vertex at origin which
contains the disc of radius $2^j$ at $y-x'$. Then the angle of $\Theta$ is at most a constant multiple of $2^{-i+j}$. Since
$|y-x'|-2^j\leq r\leq|y-x'|+2^j$, so the integrate area on $r$ is $$[|y-x'|-2^j,|y-x'|+2^j]\cap[2^{i-2},2^i]$$ and
we set it as $[r_1,r_2]$. Using $j\leq i-3$, we have the estimates $\|\psi\|_\infty\leq C2^{-i}2^{-2j}\|a\|_\infty$ and
$\|\psi'\|_\infty\leq C2^{-i}2^{-3j}\|a\|_\infty$. After making a coordinate transform back, we get the integrate area $A$
and $A'$, where $A$ is the triangle with vertices $\{y+(r_1,0),y+(r_2,0),x\}$ and $A'$ is the triangle with vertices
$\{y+(r_1,0),y+(r_2,0),x'\}$. Then we have
\[\int_A\frac{dudv}{|x_2-v|}=\int_{y_2}^{x_2}(r_2-r_1)\frac{x_2-v}{x_2-y_2}\frac{dv}{x_2-v}=r_2-r_1\leq 2^{j+1}.\]
Since we have assumed that $|x-x'|\leq 2^j\frac{1}{1000}$, so it is not necessary to restrict $|x-y|>10|x-x'|^{\frac{1}{10}}
2^{\frac{9}{10}j}$ anymore. We just need $\frac{|x_2-y_2|}{|x-y|}>\frac{1}{10}t_0$ and $2^{i-3}\leq|y-x|\leq 2^{i+1}$. At last we obtain the H\"older smoothness estimate
\begin{equation}\label{e:hold}
|K_{i,j}(y,x)-K_{i,j}(y,x')|\leq C2^{-2i}2^{-j/20}|x-x'|^{\frac{1}{20}}\omega(y).
\end{equation}
Now we take a cube $Q_n$ with side length $2^{j-s}$ and use the H\"older smoothness estimate (\ref{e:hold}) to get
\begin{equation*}
\Big|\int K_{i,j}(y,x)b_n(x)dx\Big|=\Big|\int K_{i,j}(y,x)-K_{i,j}(y,x_{Q_n})b_n(x)dx\Big|
\leq C2^{-s/20}2^{-2i}\om(y)\|a\|_\infty\|b_n\|_1
\end{equation*}
where $x_{Q_n}$ is the center of $Q_n$. For a fixed $y$ the kernel $K_{i,j}(y,x)$ is supported in $B(y,2^{i+1})$. Hence we have
\[\Big|\int K_{i,j}(y,x)B_{j-s}(x)dx\Big|\leq C\|a\|_\infty^22^{-\frac{s}{20}}2^{-2i}\om(y)\sum_{Q_n\in \mathfrak{Q}_{j-s}\atop Q_n\subset B(y,2^{i+1})}\|b_n\|_1.\]
 Then we sum over $j\leq i-3$. Note that the cubes $Q_n$ are disjoint each other, therefore we get
\begin{equation}\label{e:last}
\begin{split}
&\sum_{j\leq i-3}\sum_{Q_n\in \mathfrak{Q}_{j-s}\atop Q_n\subset B(y,2^{i+1})}\|b_n\|_1
\leq\sum_{j\leq i-3}\sum_{Q_n\in \mathfrak{Q}_{j-s}\atop Q_n\subset B(y,2^{i+1})}\frac{C\lambda|Q_n|}{\|a\|_\infty}\leq
\frac{C2^{2i}\lambda}{\|a\|_\infty},
\end{split}
\end{equation}
where the last inequality comes from all the cubes that appear in (\ref{e:last}) are contained in a disc of radius $2^{i+1}$. Hence we have
\[\Big|\sum_{j\leq{i-3}}(T_i^*T_j)_{\omega}B_{j-s}(y)\Big|=\sum_{j\leq{i-3}}\Big|\int K_{i,j}(y,x)B_{j-s}(x)dx\Big|\leq C2^{-\varepsilon s}\lambda\|a\|_\infty\omega(y).\]

For the case $2^{j-3}<|y|<2^{j+1}$. As in Lemma \ref{1:f2}, we write $K_i=K_i^1+K_i^2$ and denote
\[K_{i,j}^1=\int K_j(z-x)K_i^1(z-y)m_{x,z}a\cdot\omega(z)dz\]
and
\[K_{i,j}^2=\int K_j(z-x)K_i^2(z-y)m_{x,z}a\cdot\omega(z)dz.\]
A similar argument, which has been used  to deal with the term corresponding to $K^1_{j,j}$ in proving Lemma \ref{1:f2},
can be applied to estimate $K^1_{i,j}$. Combining the method we deal with $K^2_{j,j}$ in Lemma \ref{1:f2} and the method
we handle $K_{i,j}$ in the case $|y|>2^{j+1}$ or $|y|<2^{j-3}$, we can get the proof of the term related to
$K^2_{i,j}$. Therefore we complete the proof of Lemma \ref{e:cross1}.
\end{proof}

\vspace{4mm}
\section{Proof of Theorem $\ref{t:1}$}
In \cite{H1}, Hofmann gave a weighted $L^p$ boundedness for general singular
integral operators. We will show that the commutator $T_a$ discussed in this paper is an example of that general operator.

Before stating the theorem in \cite{H1}, let us give some notations. For an open set $\Omega$ in $\Bbb R^d$, we denote by $C^{\infty}_c(\Omega)$
the set of functions with continuous derivatives of any order and compact support in $\Omega$.
Let $\psi\in C^{\infty}_c(\{|x|<1\})$ be radial, non-trivial, have mean value zero, and be normalized so that
$\int_0^\infty|\hat{\psi}(s)|^2\frac{ds}{s}=1$, then we define $Q_sf=\psi_s\ast f$, where $\psi_s(x)=s^{-d}\psi(\frac{x}{s})$.

Let $\mathfrak{D}$ denote the space of smooth function with compact support in $\Bbb R^d$ and $\mathfrak{D'}$ be its dual space. We assume that $T$ maps $\mathfrak{D}$ to $\mathfrak{D}'$ and $T$ is associated a kernel $K(x,y)$ in the
sense that for $f,g\in C^\infty_{c}(\Bbb R^d)$ with disjoint support
\[\langle Tf,g \rangle=\iint K(x,y)f(y)g(x)dydx.\]

We will introduce some conditions similar to the conditions of $T1$ Theorem. We first suppose the kernel $K$ satisfies the size condition:
\begin{equation}\label{e:size}
|K(x,y)|\leq C_1|x-y|^{-d}.
\end{equation}
Let $\varphi\in C^\infty_c(\frac{1}{2},2)$. Then for $v\in(0,\infty)$ we set
\[T_vf(x)=\int K(x,y)\varphi\Big(\frac{|x-y|}{v}\Big)f(y)dy.\]

We introduce the \emph{weak smoothness condition} (WS):
\begin{equation}\label{e:wsmoth}
\|Q_sT_v\|_{op}+\|Q_sT^*_v\|_{op}\leq C_2\| \psi\|_1(\|\varphi\|_\infty+\|\varphi'\|_\infty)(\frac{s}{v})^{\varepsilon_0}.
\end{equation}
for some $0<\varepsilon_0\leq 1$ and $s<v$, where $T^*$ is the adjoint operator of $T$ and $\|\cdot\|_{op}$ denotes the norm of operator mapping $L^2$ to $L^2$.

As usual, we require the \emph{weak boundedness property} (WBP):
\begin{equation}\label{e:WBP}
\langle Th,\tilde{h}\rangle\leq C_3R^d(\|h\|_\infty+R\|\nabla h\|_\infty)
(\|\tilde{h}\|_\infty+R\|\nabla \tilde{h}\|_\infty).
\end{equation}
for all $h,\tilde{h}\in C^\infty_c(\R^d)$ with support in any ball of radius $R$.

To define $T1$, we impose the \emph{qualitative technical condition} (QT):
\begin{equation}\label{e:techni}
\begin{split}
\int_{|x-u|>2s}&\Big|\int\psi_s(x-z)K(z,u)dz\Big|du<\infty,\\
\int_{|x-u|>2s}&\Big|\int\psi_s(x-z)K^*(z,u)dz\Big|du<\infty,
\end{split}
\end{equation}
where $K^*(z,u)$ is the kernel of $T^*$. Let $\tilde{\psi}\in C^\infty_{c}(B(x_0,s))$ and $\int\tilde{\psi}(x)dx=0$.
 We write $1=h+(1-h)$, where $h\in C^\infty_{c}(B(x_0,4s))$ and $h\equiv1$ for $|x-x_0|\leq2s$. Then we define
\[\langle \tilde{\psi},T1\rangle=\langle \tilde{\psi},Th\rangle+\langle T^*\tilde{\psi},1-h\rangle,\]
where the second term is well defined by (\ref{e:techni}).

We consider truncations of $T$.
Let $\Phi\in C^\infty_c(-1,1)$ and $\Phi\equiv1$ on $[-\frac{1}{2},\frac{1}{2}]$. For $t<r$ and $f\in L^p$ with $1< p<\infty$, $T_{(t,r)}$ is defined as follows:
\[T_{(t,r)}f(x)=\int K(x,y)\Big[\Phi\Big(\frac{|x-y|}{r}\Big)-\Phi\Big(\frac{|x-y|}{t}\Big)\Big]f(y)dy.\]
$T_{(0,r)}$ can be defined formally:
\[T_{(0,r)}f(x)=\int K(x,y)\Phi\Big(\frac{|x-y|}{r}\Big)f(y)dy.\]

Now we need two conditions to replace the usual condition $T1$, $T^*1\in BMO$. One is the \emph{quasi-Carleson measure condition} (QCM): For any ball $B$ of radius $10\sqrt{d}t$, $1<q<\infty$, $t>0$, we have
\begin{equation}\label{e:qCarl}
\Big\|\Big(\int^t_0|Q_sT_{(0,t)}1|^2\frac{ds}{s}\Big)^\frac{1}{2}\Big\|_{L^q(B,\frac{dx}{t^d})}\leq C_4,
\end{equation}
where $\frac{dx}{t^d}$ denotes normalized Lebesgue measure.

The other is the \emph{local paraproduct type condition} (LP): For all $r>0$ and $1<q<\infty$, for all $f\in C_c^\infty(\R^d)$ with support in any ball of radius $10\sqrt{d}r$, we have
\begin{equation}\label{e:lpara}
\|\pi_rf\|_{L^q(B,\frac{dx}{r^d})}<C_5 \|f\|_\infty,
\end{equation}
where
\[\pi_rf=\int^r_0\int^t_0Q_s(Q_sT_{(t,r)}1Q^2_tf)\frac{ds}{s}\frac{dt}{t}.\]

\begin{thma}{\rm (See \cite[ Theorem 2.14]{H1})}
Suppose that $T,T^*$ and its kernel $K$ satisfies $(\ref{e:size})\sim(\ref{e:lpara})$. Then for all $\omega\in A_p$ with
$1<p<\infty$, we have
\begin{equation}
\|Tf\|_{p,\omega}\leq C\Big(\sum\limits_{1\leq i\leq5}C_i\Big)\|f\|_{p,\omega}
\end{equation}
\end{thma}

To prove Theorem \ref{t:1}, we only need to verify that the commutators $T_a$, $T^*_a$ and the kernel $L(x,y)=K(x-y)m_{x,y}(a)$ satisfies the conditions $(\ref{e:size})\sim(\ref{e:lpara})$ with $C_i$ bounded by $C\|a\|_\infty$, $1\leq i\leq5$. It is trivial to see that $|L(x,y)|\leq C\frac{\|a\|_\infty}{|x-y|^d}$. By the $L^2$ boundedness of $T_a$ (see \cite{CJ}), $\|T_a\|_{2\rightarrow2}\leq C\|a\|_{\infty}$. Then we have
\[|\langle T_ah,\widetilde{h}\rangle|\leq \|T_a\|_{2\rightarrow2}\|h\|_2\|\widetilde{h}\|_2\leq C\|a\|_\infty\|h\|_\infty \|\widetilde{h}\|_\infty R^d,\]
where $h,\widetilde{h}\in C_c^{\infty}(\R^d)$ and $h,\widetilde{h}$ support in $B(x_0,R)$. By Theorem A, the proof of the Theorem $\ref{t:1}$ follows immediately from the next four clams:

\hspace{-0.7cm}\textbf{{Clam 1}}: The operator $T_a$ satisfies the weak smooth condition $(\ref{e:wsmoth})$, which means that
$$\|Q_sT_{a,v}\|_{op}+\|Q_sT^*_{a,v}\|_{op}\leq C\|a\|_\infty\| \psi\|_1(\|\varphi\|_\infty+\|\varphi'\|_\infty)(\frac{s}{v})^{\varepsilon_0},$$
where $$T_{a,v}f(x)=\int L(x,y)\varphi\Big(\frac{|x-y|}{v}\Big)f(y)dy.$$

The proof of Clam 1 is similar to the proof of Lemma 4.3 in \cite{H1}. We just give the difference by the following lemma.
\begin{lemma}\label{aux}Let $S_1$ denote the convolution operator with the kernel $H(x)=K(x)\varphi(|x|)$, where $K$ is a Calder\'on-Zygmund convolution kernel. Then $\|Q_sS_1\|_{op}\leq Cs^{\varepsilon_0}$
for $s<1$ and $0<\varepsilon_0 <1$.
\end{lemma}

\begin{proof}
In fact, by Plancherel theorem we only need to check $|\hat{\psi}_s(\xi)\hat{H}(\xi)|
\leq Cs^{\varepsilon_0}$, for $0<\varepsilon_0<1$. We firstly give an estimate of $\hat{H}(\xi)$. Write
\[\hat{H}(\xi)=\int_{S^{d-1}}\int e^{-2\pi ir\theta\cdot\xi}K(r\theta)\varphi(r)r^{d-1}drd\theta.\]
By Van der Corput's lemma, we have
\[\Big |\int e^{-2\pi ir\theta\cdot\xi}K(r\theta)\varphi(r)r^{d-1}dr\Big|\leq\frac{C}{2\pi|\theta\cdot\xi|}.\]
On the other hand, by
$$\bigg|\int e^{-2\pi ir\theta\cdot\xi}K(r\theta)\varphi(r)r^{d-1}dr\bigg|\leq C,$$
thus it is also dominated by
$C|\theta\cdot\xi|^{-\varepsilon_0}$ for any $0<\varepsilon_0<1$. So we have $|\hat{H}(\xi)|\leq C|\xi|^{-\varepsilon_0}$. For $|s\xi|>1$, $|\hat{\psi_s}(\xi)\hat{H}(\xi)|\leq\|\psi\|_{L^1}|\xi|^{-\varepsilon_0}\leq
Cs^{\varepsilon_0}$. For $|s\xi|\leq 1$, since $$|\hat{\psi_s}(\xi)|=|\hat{\psi}(s\xi)|=\bigg|\int(e^{-2\pi is\xi
x}-1)\psi(x)dx\bigg|\leq C|s\xi|^{\varepsilon_0},$$ we have $|\hat{\psi}_s(\xi)\hat{H}(\xi)|\leq
Cs^{\varepsilon_0}$ for $0<\varepsilon_0<1$. Hence we complete the proof.
\end{proof}

\hspace{-0.7cm}\textbf{{Clam 2}}: The operator $T_a$ satisfies the technical condition $(\ref{e:techni})$.

Since $L^*(x,y)$ has the same form as $L(x,y)$,
it is sufficient to prove ($\ref{e:techni}$) for $L(x,y)$. We need to use the following  estimate (see Lemma 3 in \cite[p.~68]{CJ}):
\begin{equation}\label{e:CJ}
\iiint_{x,y,y'\in B(x_0,R)\atop |y-y'|<r}\big|(m_{x,y}a)^k-(m_{x,y'}a)^k\big|^2
dydy'dx\leq Ck^2\big(\frac{r}{R}\big)^{\frac{2}{3}}r^dR^{2d}\|a\|_\infty^{2k},
\end{equation}
where $0<r<R$, $k$ is a positive integer. Since $\psi$ has mean value zero, we have
\[\int_{|x-u|>2s}\Big|\int\psi_s(x-z)K(z-u)m_{z,u}adz\Big|du\leq H_1+H_2,\]
where
\[H_1=\int_{|x-u|>2s}\Big|\int\psi_s(x-z)(K(z-u)-K(x-u))m_{z,u}adz\Big|du,\]
and
\[H_2=\int_{|x-u|>2s}\Big|\int\psi_s(x-z)K(x-u)(m_{z,u}a-m_{x,u}a)dz|du.\]
For $H_1$, we have
\[H_1\leq\int_{|x-u|>2s}\int_{|x-z|<s}\frac{|z-x|}{s^d\cdot|x-u|^{d+1}}dzdu\leq C.\]
For $H_2$, we have
\begin{align*}H_2&=\sum_{j=0}^{+\infty}\int_{2^{j+1}s>|x-u|>2^js}\Big|
\int\psi_s(x-z)K(x-u)(m_{z,u}a-m_{x,u}a)dz\Big|du=:\sum_{j=0}^{+\infty}m_j(x).
\end{align*}
Now we consider
\begin{align*}
&\mspace{20mu}\frac{1}{|B(x_0,2^js)|}\int_{B(x_0,2^js)}m_j(x)dx\\
&\leq\frac{1}{(2^js)^dv_n}\int_{B(x_0,2^js)}\int_{2^{j+1}s>|x-u|>2^js}
\int_{|x-z|<s}\frac{C}{s^d}\frac{1}{(2^js)^d}|m_{z,u}a-m_{x,u}a|dzdudx\\
&\leq\frac{1}{(2^js)^dv_n}\frac{C}{s^d(2^js)^d}\iiint_{x,z,u\in B(x_0,2\cdot2^js)\atop |x-z|<s}|m_{z,u}a-m_{x,u}a|dzdudx\\
&\leq C2^{-\frac{1}{3}j},
\end{align*}
where the third inequality follows from H\"older's inequality and ($\ref{e:CJ}$). Note that the constant $C$ is independent of $s$, so $m_j(x)\leq Mm_j(x)\leq C\cdot2^{-\frac{1}{3}j}$. We hence get
\[H_2\leq\sum_{j=0}^{+\infty}C\cdot2^{-\frac{1}{3}j}<C.\]

\hspace{-0.7cm}\textbf{{Clam 3}}: The operator $T_a$ satisfies the quasi-Carleson measure condition $(\ref{e:qCarl})$.

By dilation invariance we may take $t=1$. Suppose $B$ is a ball of radius $10\sqrt{d}$ with center $x_0$. We have
\[\langle T_{(0,1)}f,g\rangle=\iint K(x-y)m_{x,y}(a)\Phi(|x-y|)f(y)g(x)dydx.\]
Here and in the sequel we still use the notation $T_{(0,1)}$. Choose $\eta\in C_c^\infty(\Bbb R^d)$, such that $\eta(x)=1$ on $2B(x_0,10\sqrt{d})$ and $\eta(x)=0$ on $(4B(x_0,10\sqrt{d}))^c$. By the support of $Q_sT_{(0,1)}1$, we have
\begin{equation}\label{e:3_1}
\Big(\int_B\Big(\int_0^1\big|Q_sT_{(0,1)}1(x)\big|^2\frac{ds}{s}\Big)^\frac{q}{2}dx\Big)^\frac{1}{q}=
\Big(\int_B\Big(\int_0^1\big|Q_sT_{(0,1)}\eta(x)\big|^2\frac{ds}{s}\Big)^\frac{q}{2}dx\Big)^\frac{1}{q}.
\end{equation}
By Littlewood-Paley theory (See \cite{L}), ($\ref{e:3_1}$) is majorized by $(\int_B(T_{(0,1)}\eta(x))^qdx)^{\frac{1}{q}}$. If the
operator with kernel $K(\cdot)\Phi(|\cdot|)$ is bounded on $L^2$, then by Christ's result in \cite{CJ},
$\|T_{(0,1)}\|_{q\rightarrow q}\leq C\|a\|_\infty$ for all $1<q<+\infty$. Hence ($\ref{e:3_1}$) is bounded.

Indeed, it is easy to check that $K(\cdot)\Phi(|\cdot|)$ is still a Calder\'on-Zygmund convolution kernel. Note that
$$\widehat{K\Phi}(\xi)=\hat{K}\ast\hat{\Phi}(\xi)
=\int_{\mathbb{R}^n}\hat{K}(y)\hat{\Phi}(|\xi-y|)dy$$
is bounded since  $K$ is a Calder\'on-Zygmund convolution kernel and $\hat{\Phi}$ is a Schwartz function. So, the operator with the kernel  $K(\cdot)\Phi(|\cdot|)$, initially defined on Schwartz class, has a bounded extension to an operator mapping $L^2(\mathbb{R}^n)$ to itself.

\hspace{-0.7cm}\textbf{{Clam 4}}: The operator $T_a$ satisfies the local paraproduct condition $(\ref{e:lpara})$.

By dilation invariance, we may take $r=1$. Let $f,g\in C_c^\infty(\Bbb R^d)$ with support in
$B(x_0,10\sqrt{d})$. Fix $1<q<\infty$, by duality we only need to prove
\begin{equation}\label{e:5_1}
\Big|\int_\varepsilon^1\int_\delta^t\langle Q_sT_{(t,1)}1Q_t^2f,Q_sg\rangle\frac{ds}{s}\frac{dt}{t}\Big|\leq C\|a\|_\infty\|f\|_\infty\|g\|_{q'}.
\end{equation}
We write
\begin{equation*}
\langle Q_sT_{(t,1)}1Q_t^2f,Q_sg\rangle=\langle Q_tf,Q_t(Q_sT_{(t,1)}1Q_sg)\rangle.
\end{equation*}

Consider $Q_t(Q_sT_{(t,1)}1)(x)$, we can replace $1$ by $\eta$, where $\eta\in C_c^\infty(\Bbb R^d)$ with $\eta\equiv1$ on
 $B(x,3t)$ and $\eta\equiv0$ on $(B(x,4t))^c$. By H\"older's inequality,
\Be
\begin{split}\label{e:QsTt}
&\mspace{20mu}|Q_t(Q_sT_{(t,1)}1Q_sg)(x)|\\
&\leq C\Big(\fr{1}{t^n}\int_{|x-z|\leq t}|Q_sT_{(t,1)}\eta(z)|^{q'_1}dz\Big)^{\fr{1}{q'_1}}
\Big(\fr{1}{t^n}\int_{|x-z|\leq t}|Q_sg(z)|^{q_1}dz\Big)^{\fr{1}{q_1}}\\
&\leq C(M(|Q_sg|^{q_1})(x))^{\fr{1}{q_1}}\Big(\fr{1}{t^n}\int_{|x-z|\leq t}|Q_sT_{(t,1)}\eta(z)|^{q'_1}dz\Big)^{\fr{1}{q'_1}},
\end{split}
\Ee
where we choose $1<q_1<2$.
Since $\|Q_sT_{a,v}f\|_\infty\leq C\|T_{a,v}f\|_\infty\leq C\|a\|_\infty\|f\|_\infty$, by Clam 1 and using interpolation we have
\begin{equation}\label{e:5_2}
\|Q_sT_{a,v}\|_{p\rta p}\leq C\|a\|_\infty\big(\frac{s}{v}\big)^{\varepsilon(p)}
\end{equation}
for all $2\leq p<+\infty$. Make a smooth partition of unity and write
\[\Phi\big(\frac{\rho}{t}\big)-\Phi(\rho)=\sum_{j:t\leq2^jt\leq4}\tilde{\varphi}\big(\frac{\rho}{2^jt}\big)\Big
(\Phi\big(\frac{\rho}{t}\big)-\Phi(\rho)\Big),\]
where $\tilde{\varphi}\in C_c^\infty(\frac{1}{2},2)$ and
$\sum_{j=-\infty}^{\infty}\tilde\varphi({\frac{\rho}{2^j})}\equiv1$ for all $\rho>0$. We define
\[T_{a,j}f(x)=\int L(x,y)\tilde{\varphi}\Big(\frac{|x-y|}{2^jt}\Big)\Big(\Phi\big(\frac{|x-y|}{t}\big)-\Phi(x-y)\Big)f(y)dy,\]
then $T_{(t,1)}=\sum_{\frac{t}{4}\leq 2^jt\leq2}T_{a,j}$. Applying Minkowski inequality and (\ref{e:5_2}), we have
\Be\label{e:T(t,1)}
\begin{split}
\Big(\fr{1}{t^n}\int_{|x-z|\leq t}|Q_sT_{(t,1)}\eta(z)|^{q'_1}dz\Big)^{\fr{1}{q'_1}}
&\leq\Big(\fr{1}{t^d}\Big)^{\fr{1}{q'_1}}\sum_{\frac{t}{4}\leq 2^jt\leq2}||Q_sT_{a,j}\eta||_{q'_1}\\
&\leq\Big(\fr{1}{t^d}\Big)^{\fr{1}{q'_1}}\sum_{\frac{t}{4}\leq 2^jt\leq2}C\|a\|_\infty\Big(\fr{s}{2^jt}\Big)^{\eps(q'_1)}\|\eta\|_{q'_1}\\
&\leq C\|a\|_\infty\Big(\fr{s}{t}\Big)^{\eps(q'_1)}.
\end{split}
\Ee
By estimates (\ref{e:QsTt}), (\ref{e:T(t,1)}) and H\"older's inequality, the left side of (\ref{e:5_1}) is bounded by
\Be
\begin{split}
&\mspace{20mu}\int_\varepsilon^1\int_\delta^t
\Big(\frac{s}{t}\Big)^{\varepsilon(q'_1)}
\int_{\R^n}\big(M(|Q_sg|^{q_1})(x)\big)^{\fr{1}{q_1}}|Q_tf(x)|dx
\frac{ds}{s}\frac{dt}{t}\\
&\leq C\|a\|_\infty\int_\varepsilon^1\int_\delta^t\Big(\frac{s}{t}\Big)^{\varepsilon(q'_1)}\|Q_sg\|_{2,\fr{1}{\om}}
\|Q_tf\|_{2,\om}
\frac{ds}{s}\frac{dt}{t}.
\end{split}
\Ee
By using H\"older's inequality again, the last term above is majorized by
\begin{equation}\label{e:5_3}
C\|a\|_\infty\bigg(\int_\varepsilon^1\int_\delta^t\Big(\frac{s}{t}\Big)^{\varepsilon(q'_1)}
\|Q_tf\|_{2,\om}^2\frac{ds}{s}\frac{dt}{t}\bigg)^\frac{1}{2}
\bigg(\int_\varepsilon^1\int_\delta^t\Big(\frac{s}{t}\Big)^{\varepsilon(q'_1)}
\|Q_tg\|_{2,\fr{1}{\om}}^2\frac{ds}{s}\frac{dt}{t}\bigg)^\frac{1}{2}.
\end{equation}
Firstly,  let us consider the first integral factor in (\ref{e:5_3}).
Note that
$$\int_\delta^t\Big(\frac{s}{t}\Big)^{\varepsilon(q'_1)}\frac{ds}{s}\le C,$$
hence  by weighted Littlewood-Paley theory (See \cite{K}), we have
\[\bigg(\int_\varepsilon^1\int_\delta^t\Big(\frac{s}{t}\Big)^{\varepsilon(q'_1)}
\|Q_tf\|_{2,\om}^2\frac{ds}{s}\frac{dt}{t}\bigg)^\frac{1}{2}\le C\Big(\int_{\varepsilon}^1\|Q_tf\|_{2,\om}^2\frac{dt}{t}\Big)^{\frac{1}{2}}\leq C\|f\|_{2,\om}.\]
Then using the same method for the other factor, we get
\[\Big(\int_\delta^1\int_s^1\big(\frac{s}{t}\big)^{\varepsilon(q'_1)}\|Q_sg\|_{2,\fr{1}{\om}}^2\frac{dt}{t}\frac{ds}{s}\Big)^\frac{1}{2}\leq
C\Big(\int_\delta^1\|Q_sg\|_{2,\fr{1}{\om}}^2\frac{ds}{s}\Big)^\frac{1}{2}\leq C\|g\|_{2,\fr{1}{\om}}.\]
Therefore (\ref{e:5_1}) is controled by
$\|a\|_\infty\|f\|_{2,\om}\|g\|_{2,\fr{1}{\om}}.$
By extrapolation (See \cite{L2}), we can replace this bound by
$\|a\|_\infty\|f\|_{q,\om}\|g\|_{q',\fr{1}{\om}}$, for any $1<q<\infty$.
Now since $\om$ is a Muckenhoupt weight, we can replace the bound by $\|a\|_\infty\|f\|_q\|g\|_{q'}$ by setting $\om\equiv1$. Since $f\in C_c^\infty(\Bbb R^d)$ with compact support, we have $\|f\|_{q}\leq C\|f\|_{\infty}$. Hence we complete the proof of Theorem \ref{t:1}.

\subsection*{Acknowledgment}
The authors would like to express their deep gratitude to the referee for his/her very careful reading, important comments and valuable suggestions.

\vspace{4mm}

\bibliographystyle{amsplain}

\end{document}